\documentclass{amsart}

\usepackage{amssymb,latexsym}
\usepackage{graphicx}
\usepackage[all,2cell]{xy}
\SelectTips{eu}{12}

\newcommand{\Z}{\mathbb{Z}}

\newcommand{\R}{\mathbb{R}}
\newcommand{\C}{\mathbb{C}}
\newcommand{\CP}{\mathbb{C}\mathbf{P}}
\newcommand{\SO}{\mathrm{SO}}
\newcommand{\OO}{\mathrm{O}}
\newcommand{\codim}{\mathrm{co}\!\dim}

\newcommand{\rank}{\mathrm{rank}}

\newtheorem{lemma}{Lemma}[section]

\newtheorem{proposition}[lemma]{Proposition}
\newtheorem{theorem}[lemma]{Theorem}
\newtheorem{corollary}[lemma]{Corollary}

\theoremstyle{definition}
\newtheorem{remark}[lemma]{Remark}
\newtheorem{definition}[lemma]{Definition}

\begin{document}
\parindent0em
\setlength\parskip{.1cm}
\title[Intersection homology of linkage spaces]{Intersection homology of linkage spaces in odd dimensional Euclidean space}
\author{Dirk Sch\"utz}
\address{Department of Mathematical Sciences\\ University of Durham\\ Science Laboratories\\ South Rd\\ Durham DH1 3LE\\ United Kingdom}
\email{dirk.schuetz@durham.ac.uk}
\begin{abstract}
We consider the moduli spaces $\mathcal{M}_d(\ell)$ of a closed linkage with $n$ links and prescribed lengths $\ell\in \R^n$ in $d$-dimensional Euclidean space. For $d>3$ these spaces are no longer manifolds generically, but they have the structure of a pseudomanifold.

We use intersection homology to assign a ring to these spaces that can be used to distinguish the homeomorphism types of $\mathcal{M}_d(\ell)$ for a large class of length vectors. These rings behave rather differently depending on whether $d$ is even or odd, with the even case having been treated in an earlier paper. The main difference in the odd case comes from an extra generator in the ring which can be thought of as an Euler class of a startified bundle.
\end{abstract}
\maketitle

\section{Introduction}
In this paper we continue our studies of the moduli spaces of closed $n$-gon linkages in high dimensional Euclidean space. These spaces are determined by a \em length vector \em $\ell\in \R^n$ such that all entries are positive. More precisely, the moduli space we are interested in is
\begin{eqnarray*}
 \mathcal{M}_d(\ell)&=&\left\{ (x_1,\ldots,x_n)\in (S^{d-1})^n\,\left|\, \sum_{i=1}^n \ell_ix_i=0\right\}\right/ \SO(d)
\end{eqnarray*}
where $\SO(d)$ acts diagonally on the product of spheres. A natural question is how the topology of $\mathcal{M}_d(\ell)$ depends on $\ell$, and one of the first observations is that $\R^n$ is divided into finitely many chambers such that length vectors in the same chamber lead to homeomorphic moduli spaces. One may then ask whether length vectors from different chambers (up to permutation of coordinates) have different moduli spaces.

In the planar case $d=2$ Walker \cite{walker} conjectured that the cohomology ring of these spaces is enough to distinguish them, which was then confirmed by Farber, Hausmann and the author in \cite{fahasw, schuew}. Furthermore, in \cite{fahasw} this was also shown for $d=3$, with the single exception of $n=4$ where for two different chambers the moduli space is the $2$-sphere. Indeed, it follows from the work of Schoenberg \cite{schoen} that for $n=d+1$ each non-empty moduli space is a sphere, and for $n\leq d$ the non-empty ones are discs. The homology calculations of \cite{schuet} indicate that for $n\geq d+2$ the topology does depend on $\ell$, but they also show that homology and cohomology are not enough to distinguish them. In \cite{schint} the author used intersection homology to distinguish moduli spaces of a large class of length vectors for $d\geq 4$ even. Here we show that this approach also works for $d\geq 5$ odd.

The main theorem we thus get for the topology of moduli spaces is the following.

\begin{theorem}
 \label{main_theorem}
Let $d\geq 2$, $\ell,\ell'\in \R^n$ be generic, $d$-regular length vectors. If $\mathcal{M}_d(\ell)$ and $\mathcal{M}_d(\ell')$ are homeomorphic, then $\ell$ and $\ell'$ are in the same chamber up to a permutation.
\end{theorem}

For the precise definition of generic and $d$-regular we refer the reader to Section \ref{sec_baisc}. As mentioned above, the remaining case is when $d\geq 5$ is odd. For $d\geq 4$ the moduli spaces are no longer manifolds, so in \cite{schint} a substitute for the cohomology ring using intersection homology was defined. For $d$ even specific generators were found and it was shown that this ring is an exterior face ring similar to the situation when $d=2$.

For odd $d$ this ring has an extra generator, which makes the determination of the ring more difficult. However, this is the same situation as in the case $d=3$ where the cohomology ring was calculated by Hausmann and Knutson \cite{hauknu}. While we do not determine the ring completely, we do obtain enough information to mimick the proof used in \cite{fahasw} to get the result for $d=3$, which relied on the cohomology description of \cite{hauknu}. A crucial observation is that the extra generator can be thought of as an Euler class of a certain stratified bundle over $\mathcal{M}_d(\ell)$ from which its multiplication with the other generators can be deduced.

Similar results to Theorem \ref{main_theorem} have been obtained in \cite{fahasc} and by Farber and Fromm in \cite{farfro} for chain spaces and free polygon spaces, respectively. These spaces are closed manifolds for generic $\ell$ and all $d\geq 2$, and the proofs do not rely on distinguishing between even and odd $d$. It may be possible to give a unified proof of Theorem \ref{main_theorem} by using $\Z/2$ coefficients throughout.

One could ask whether the condition of $d$-regularity is necessary for $n>d+1$, as it is not necessary for $d=2,3$. It seems unlikely though that this can be attacked using intersection homology.

\section{Basic properties of linkage spaces}
\label{sec_baisc}
We define the \em chain space \em of a length vector $\ell$ as
\begin{eqnarray*}
 \mathcal{C}_d(\ell)&=&\left\{(x_1,\ldots,x_{n-1})\in (S^{d-1})^{n-1}\,\left|\, \sum_{i=1}^{n-1}\ell_i x_i=-\ell_ne_1\right\}\right.
\end{eqnarray*}
where $e_1=(1,0,\ldots,0)\in \R^d$ is the usual first coordinate vector. If we let $\SO(d-1)$ act on $S^{d-1}$ by fixing the first coordinate, we see that $\SO(d-1)$ acts diagonally on $\mathcal{C}_d(\ell)$ and
\begin{eqnarray*}
 \mathcal{M}_d(\ell)&\cong & \mathcal{C}_d(\ell)/\SO(d-1).
\end{eqnarray*}
We also define
\begin{eqnarray*}
 \mathcal{N}_d(\ell)&=&\left\{ (x_1,\ldots,x_n)\in (S^{d-1})^n\,\left|\, \sum_{i=1}^n \ell_ix_i=0\right\}\right/ \OO(d)
\end{eqnarray*}
so that $\mathcal{N}_d(\ell)\cong \mathcal{M}_d(\ell)/(\Z/2)$.

\begin{definition}
 Let $\ell\in \R^n$ be a length vector. A subset $J\subset \{1,\ldots,n\}$ is called \em $\ell$-short\em, if
\begin{eqnarray*}
 \sum_{j\in J}\ell_j&<& \sum_{i\notin J}\ell_i.
\end{eqnarray*}
It is called \em $\ell$-long\em, if the complement is $\ell$-short. If every such subset is either $\ell$-short or $\ell$-long, the length vector is called \em generic\em.

A length vector $\ell=(\ell_1,\ldots,\ell_n)$ is called \em ordered\em, if $\ell_1\leq \ell_2\leq\cdots\leq \ell_n$.
\end{definition}

After permuting the coordinates we can always assume that $\ell$ is ordered.

If $\ell\in \R^n$ is ordered and $k\leq n-3$, we write
\begin{eqnarray*}
 \mathcal{S}_k(\ell)&=& \{J\subset\{1,\ldots,n\}\,|\,n\in J, \, |J|=k+1, J \mbox{ is }\ell\mbox{-short}\}.
\end{eqnarray*}

The cardinality of these sets is denoted by
\begin{eqnarray*}
 a_k(\ell)&=&|\mathcal{S}_k(\ell)|.
\end{eqnarray*}

If $J\subset\{1,\ldots,n\}$, we define the hyperplane
\begin{eqnarray*}
 H_J&=&\left\{(x_1,\ldots,x_n)\in \R^n\,\left|\, \sum_{j\in J}x_j= \sum_{j\notin J}x_j\right\}\right.
\end{eqnarray*}
and let
\begin{eqnarray*}
 \mathcal{H}&=& \R^n_{>0} - \bigcup_{J\subset\{1,\ldots,n\}}H_J,
\end{eqnarray*}
where $\R^n_{>0}=\{(x_1,\ldots,x_n)\in \R^n\,|\,x_i>0\}$. Then $\mathcal{H}$ has finitely many components, which we call \em chambers\em. It is clear that a length vector $\ell$ is generic if and only if $\ell\in\mathcal{H}$. 

It is shown in \cite{hausma} that if $\ell$ and $\ell'$ are in the same chamber, then $\mathcal{C}_d(\ell)$ and $\mathcal{C}_d(\ell')$ are $\OO(d-1)$-equivariantly diffeomorphic. In particular, $\mathcal{M}_d(\ell)$ and $\mathcal{M}_d(\ell')$ are homeomorphic.

It is easy to see that two ordered generic length vectors $\ell$, $\ell'$ are in the same chamber if and only if $\mathcal{S}_k(\ell)=\mathcal{S}_k(\ell')$ for all $k=0,\ldots,n-3$.

\begin{definition}
 Let $\ell\in \R^n$ be a length vector and $d\geq 2$. Then $\ell$ is called $d$-regular, if
\begin{eqnarray*}
 \bigcap_{J\in \mathcal{L}^d(\ell)}J &\not=& \emptyset
\end{eqnarray*}
where $\mathcal{L}^d(\ell)$ are the subsets $J\subset\{1,\ldots,n\}$ with $d-1$ elements that are $\ell$-long. If $\mathcal{L}^d(\ell)=\emptyset$, we let the intersection above be $\{1,\ldots,n\}$.
\end{definition}

So for a length vector to be \em not \em $d$-regular, we need a long subset $J\subset \{1,\ldots,n\}$ with $d-1$ elements such that $J$ does not contain an element $m$ with $\ell_m$ maximal. Because then we can replace any element of $J$ with $m$ to get another $\ell$-long subset with $d-1$ elements, and the intersection of these sets will be empty.

If $\ell$ is ordered, then $\ell$ is $d$-regular if and only if $\{n-d+1,n-d+2,\ldots,n-1\}$ is not $\ell$-long. For a generic length vector this is equivalent to $\mathcal{S}_{n-d}(\ell)=\emptyset$.

It follows from the definition that every length vector with $n\geq 2$ is $2$-regular. Furthermore, there is only one chamber up to permutation which is not $3$-regular, namely the one containing\footnote{Technically, this $\ell$ is not a length vector because of $0$-entries. We interpret a $0$-entry in a length vector as $\varepsilon>0$ so small that decreasing it does not change the chamber.} $\ell=(0,\ldots,0,1,1,1)$. In \cite{fahasc}, $4$-regular was called normal.

In the case $d=n-1$, there are only two generic $d$-regular length vectors $\ell\in \R^n$ up to permutation, namely $\ell=(1,\ldots,1,n-2)$ and $\ell'=(0,\ldots,0,1)$. If $n$ is large compared to $d$, $d$-regularity gets more common, and we would expect the ratio of all $d$-regular length vectors in $\R^n$ by all length vectors in $\R^n$ to converge to $1$.

For $d=2,3$ and $\ell$ generic, the spaces $\mathcal{M}_d(\ell)$ are obtained as a quotient space from $\mathcal{C}_d(\ell)$ using a free action, so they are closed manifolds. For $d\geq 4$ this is no longer the case and we get different orbit types. Let $x\in \mathcal{M}_d(\ell)$ be represented by $(x_1,\ldots,x_n)\in (S^{d-1})^n$. If we think of this as a $(d\times n)$-matrix, the rank of this matrix does not depend on the representative of $x$.

\begin{definition}
 Let $\ell\in \R^n_{>0}$ be a generic length vector and $x\in\mathcal{M}_d(\ell)$. Then the \em rank of $x$\em, $\rank(x)$, is the rank of a $(d\times n)$-matrix representing $x$.
\end{definition}

For $k< d$ we have an inclusion $\mathcal{N}_k(\ell)\to \mathcal{M}_d(\ell)$ and $\mathcal{N}_k(\ell)$ are exactly those points with rank at most $k$. It was shown in \cite{schint} that these subsets form a stratification of $\mathcal{M}_d(\ell)$ and $\mathcal{M}_d(\ell)$ is a pseudomanifold for $n\geq d+1$ . Note that $\SO(d-1)$ acts freely on points of rank $d-1$ in $\mathcal{C}_d(\ell)$, so that the regular set is $\mathcal{M}_d(\ell)-\mathcal{N}_{d-2}(\ell)$.

The stratification we will look at is therefore given by
\[
 \emptyset \subset \mathcal{N}_2(\ell)\subset \mathcal{N}_3(\ell)\subset\cdots \subset \mathcal{N}_{d-2}(\ell)\subset \mathcal{M}_d(\ell).
\]
The singular strata are therefore given by $\mathcal{N}_k(\ell)-\mathcal{N}_{k-1}(\ell)$ for $k=2,\ldots,d-2$ and they can easily seen to be connected.  The dimension of $\mathcal{M}_d(\ell)$ (and $\mathcal{N}_d(\ell)$) for $n\geq d$ is given by
\begin{eqnarray*}
 \mathbf{d}^n_d&=&(n-3)(d-1)-\frac{(d-2)(d-3)}{2},
\end{eqnarray*}
see \cite{schuet}.

We need to recall a few basics from intersection homology.

\begin{definition}
 Let $X$ be a stratified pseudomanifold, a (general) \em perversity \em is a function
\[
 \mathbf{p}\colon \{\mbox{singular strata of }X\} \to \Z.
\]
\end{definition}
For our purposes, this mainly means functions $\mathbf{p}\colon \{2,\ldots,d-2\}\to \Z$ where $k\in \{2,\ldots,d-2\}$ corresponds to the stratum $\mathcal{N}_{d-k}(\ell)-\mathcal{N}_{d-k-1}(\ell)$. Also, we are mostly interested in \em Goresky-MacPherson perversities\em, which for us means that
\begin{eqnarray*}
 \mathbf{p}(2) &\leq & 2(n-d)-1 \\
 \mathbf{p}(k+1)-\mathbf{p}(k) & \leq & n-d+k,
\end{eqnarray*}
compare \cite{schint}. The top perversity $\mathbf{t}_n=\mathbf{t}$ is then given by
\begin{eqnarray*}
 \mathbf{t}(k)&=& \mathbf{c}^n_{d,k}-2
\end{eqnarray*}
for all $k\in\{2,\ldots,d-2\}$, where $\mathbf{c}^n_{d,k}$ denotes the codimension of the stratum $\mathcal{N}_{d-k}(\ell)-\mathcal{N}_{d-k-1}(\ell)$, and which is given by
\begin{eqnarray*}
 \mathbf{c}_{d,k}^n & = & k(n-d) + \frac{k(k-1)}{2}.
\end{eqnarray*}
The perversities we need for $\mathcal{M}_d(\ell)$ are given by $\mathbf{p}_j$ for $j=0,\ldots, n-d-1$ via the formula
\begin{eqnarray*}
 \mathbf{p}_j(k)&=&j\cdot k.
\end{eqnarray*}
It is easy to check that these perversities are Goresky-MacPherson, and that $\mathbf{p}_i+\mathbf{p}_j=\mathbf{p}_{i+j}$ for $i+j \leq n-d-1$.

Given a perversity $\mathbf{p}$, we denote the resulting intersection homology by
\[
 I^\mathbf{p}H_\ast(X)
\]
where the definition is as in \cite{gormac}.

For each perversity we are usually only interested in the intersection homology in a particular degree, and this group is $$I^{\mathbf{p}_j}H_{\mathbf{d}^{n-j}_d}(\mathcal{M}_d(\ell)).$$
Note that $\mathbf{p}_0$ is the zero perversity, so this group contains the fundamental class of $\mathcal{M}_d(\ell)$ for $j=0$. Similarly, if we form a length vector $\ell_J\in \R^{n-|J|}$ from an ordered length vector $\ell$ using a subset $J\subset \{1,\ldots,n-1\}£$ by linking together the elements of $J$ with $n$, the group $I^{\mathbf{p}_j}H_{\mathbf{d}^{n-j}_d}(\mathcal{M}_d(\ell))$ contains the fundamental class of $\mathcal{M}_d(\ell_J)$ for $j=|J|$. Note that we need $|J|\leq n-d-1$ in order for $\mathcal{M}_d(\ell_J)$ to be a pseudomanifold. The condition of $d$-regularity of $\ell$ implies that $\mathcal{M}_d(\ell_J)=\emptyset$ for each $J\subset \{1,\ldots,n-1\}$ with $|J|\geq n-d$.

The following result is proven in \cite{schint}.

\begin{proposition}\label{odd_generators}
 Let $d\geq 5$ be odd, $n\geq d+2$ and $\ell\in \R^n$ a generic length vector with $\mathcal{M}_d(\ell)\not=\emptyset$. Then
\begin{eqnarray*}
 I^{\mathbf{p}_1}H_{\mathbf{d}^{n-1}_d}(\mathcal{M}_d(\ell))&\cong & \Z^{1+a_1(\ell)}.
\end{eqnarray*}

\end{proposition}

If for every $j\in \{1,\ldots,n-1\}$ with $\{j,n\}$ $\ell$-short we set
\begin{eqnarray*}
 \ell^j_+&=&(\ell_1,\ldots,\hat{\ell}_j,\ldots,\ell_n+\ell_j)\,\,\,\in\,\,\, \R^{n-1},
\end{eqnarray*}
then the fundamental classes $[\mathcal{M}_d(\ell^j_+)]\in I^{\mathbf{p}_1}H_{\mathbf{d}^{n-1}_d}(\mathcal{M}_d(\ell))$ are linearly independent (assuming that $\ell$ is ordered), as shown in \cite{schint}. The notation $\hat{\ell}_j$ is supposed to indicate that this entry is omitted. It remains to find one more generator for this group.

\section{Stratified bundles over linkage spaces}
\label{sec_bundles}
Recall that $\mathcal{M}_d(\ell)$ can be viewed as $\mathcal{C}_d(\ell)/\SO(d-1)$ where $\SO(d-1)$ acts diagonally on the left of $(S^{d-1})^n$, fixing the first coordinate in $\R^d$. We now form the space
\begin{eqnarray*}
 \overline{\mathcal{M}}_d(\ell)&=&\mathcal{C}_d(\ell)\times_{\SO(d-1)}\R^{d-1}
\end{eqnarray*}
which is the quotient space of $\mathcal{C}_d(\ell)\times \R^{d-1}$ using the equivalence relation given by $(x,v)\sim (Ax,Av)$ for $A\in \SO(d-1)$, $x\in \mathcal{C}_d(\ell)$ and $v\in \R^{d-1}$ with the standard action of $\SO(d-1)$ on $\R^{d-1}$.

There is a projection $p\colon \overline{\mathcal{M}}_d(\ell)\to \mathcal{M}_d(\ell)$ given by $p([x,v])=[x]$, which can be viewed as a stratified fibre bundle, see Remark \ref{rem_stratfibbun} below. For now it will be good enough that $\overline{\mathcal{M}}_d(\ell)$ is a pseudomanifold, so we begin by finding the right stratification.

For $k=2,\ldots,d-2$ let
\begin{eqnarray*}
 \overline{\mathcal{N}}_{d-k}(\ell)&=&p^{-1}(\mathcal{N}_{d-k}(\ell)).
\end{eqnarray*}
Note that if $[x,v]\in \overline{\mathcal{N}}_{d-k}(\ell)$, then there is $A\in \SO(d-1)$ with $Ax\in \mathcal{C}_{d-k}(\ell)$ and $Av\in \R^{d-k}$, where $\R^{d-k}\subset \R^{d-1}$ so that the last $k-1$ coordinates are $0$. Let
\begin{eqnarray*}
 \overline{\mathcal{N}}_{d-k}^{d-k-1}(\ell)&=&\{[x,v]\in \overline{\mathcal{M}}_d(\ell)\,|\, x\in \mathcal{C}_{d-k}(\ell), \, v\in \R^{d-k-1}\}.
\end{eqnarray*}
Then
\[
 \overline{\mathcal{N}}_{d-k}^{d-k-1}(\ell) \subset \overline{\mathcal{N}}_{d-k}(\ell)\subset \overline{\mathcal{N}}_{d-k+1}^{d-k}(\ell)
\]
for $k=3,\ldots,d-2$ and
\[
 \overline{\mathcal{N}}_{d-2}^{d-3}(\ell)\subset \overline{\mathcal{N}}_{d-2}(\ell)\subset \overline{\mathcal{M}}_d(\ell).
\]
Also, $p|\colon \overline{\mathcal{M}}_d(\ell)-\overline{\mathcal{N}}_{d-2}(\ell)\to \mathcal{M}_d(\ell)-\mathcal{N}_{d-2}(\ell)$ is a vector bundle, and one checks that $\overline{\mathcal{M}}_d(\ell)$ is a pseudomanifold with the stratification
\[
 \emptyset \subset \overline{\mathcal{N}}_2^1(\ell) \subset \overline{\mathcal{N}}_2(\ell) \subset \overline{\mathcal{N}}_3^2(\ell) \subset \cdots \subset \overline{\mathcal{N}}_{d-2}^{d-3}(\ell) \subset \overline{\mathcal{N}}_{d-2}(\ell)\subset \overline{\mathcal{M}}_d(\ell).
\]
The codimensions of the strata are
\begin{eqnarray*}
 \codim(\overline{\mathcal{N}}_{d-k}(\ell)-\overline{\mathcal{N}}_{d-k}^{d-k-1}(\ell))&=& \mathbf{c}^n_{d,k}+k-1\\
 \codim(\overline{\mathcal{N}}_{d-k}^{d-k-1}(\ell)-\overline{\mathcal{N}}_{d-k-1}(\ell))&=& \mathbf{c}^n_{d,k}+k
\end{eqnarray*}
for $k=2,\ldots,d-2$.

Recall the perversities $\mathbf{p}_r$ for $\mathcal{M}_d(\ell)$ and $0\leq r\leq n-d-1$ given by
\begin{eqnarray*}
 \mathbf{p}_r(\mathcal{N}_{d-k}(\ell)-\mathcal{N}_{d-k-1}(\ell))&=& k\cdot r,
\end{eqnarray*}
$k=2,\ldots,d-2$. Similarly, we define the perversities $\mathbf{q}_r$ for $\overline{\mathcal{M}}_d(\ell)$ by
\begin{eqnarray*}
 \mathbf{q}_r(\overline{\mathcal{N}}_{d-k}(\ell)-\overline{\mathcal{N}}_{d-k}^{d-k-1}(\ell))&=&r\cdot k -1\\
 \mathbf{q}_r(\overline{\mathcal{N}}_{d-k}^{d-k-1}(\ell)-\overline{\mathcal{N}}_{d-k-1}(\ell))&=& r\cdot k.
\end{eqnarray*}

The inclusion $i\colon \mathcal{M}_d(\ell)\to \overline{\mathcal{M}}_d(\ell)$ given by $i([x])=[x,0]$ is stratum preserving, as $i(\mathcal{N}_{d-k}(\ell)-\mathcal{N}_{d-k-1}(\ell))\subset \overline{\mathcal{N}}_{d-k}^{d-k-1}(\ell)- \overline{\mathcal{N}}_{d-k-1}(\ell)$ and since
\begin{multline*}
 \mathbf{p}_r(\mathcal{N}_{d-k}(\ell)-\mathcal{N}_{d-k-1}(\ell)) - \codim(\mathcal{N}_{d-k}(\ell)-\mathcal{N}_{d-k-1}(\ell)) \\
\shoveright{\,\,\,=\,\,\, k\cdot r - \mathbf{c}_{d-k}^n \,\,\,=\,\,\, k(r+1)-\mathbf{c}^n_{d,k}-k} \\ 
\,\,\, = \,\,\, \mathbf{q}_{r+1}(\overline{\mathcal{N}}_{d-k}^{d-k-1}(\ell)-\overline{\mathcal{N}}_{d-k-1}(\ell)) - \codim(\overline{\mathcal{N}}_{d-k}^{d-k-1}(\ell)-\overline{\mathcal{N}}_{d-k-1}(\ell))
\end{multline*}
the inclusion induces a map on intersection homology
\[
 i_\ast\colon I^{\mathbf{p}_r}H_\ast(\mathcal{M}_d(\ell)) \to I^{\mathbf{q}_{r+1}}H_\ast(\overline{\mathcal{M}}_d(\ell))
\]
by \cite[Lm.5.2]{schint}.

The projection $p\colon \overline{\mathcal{M}}_d(\ell)\to \mathcal{M}_d(\ell)$ satisfies
\begin{eqnarray*}
 p(\overline{\mathcal{N}}_{d-k}^{d-k-1}(\ell)-\overline{\mathcal{N}}_{d-k-1}(\ell))&\subset & \mathcal{N}_{d-k}(\ell)-\mathcal{N}_{d-k-1}(\ell)\\
 p(\overline{\mathcal{N}}_{d-k}(\ell)-\overline{\mathcal{N}}_{d-k}^{d-k-1}(\ell)) &=& \mathcal{N}_{d-k}(\ell)-\mathcal{N}_{d-k-1}(\ell)
\end{eqnarray*}
with
\begin{multline*}
 \mathbf{q}_{r+1}(\overline{\mathcal{N}}_{d-k}(\ell)-\overline{\mathcal{N}}_{d-k}^{d-k-1}(\ell))- \codim(\overline{\mathcal{N}}_{d-k}(\ell)-\overline{\mathcal{N}}_{d-k}^{d-k-1}(\ell)) \\
\shoveright{(r+1)k-1 - \mathbf{c}^n_{d,k} - k + 1 \,\,\, = \,\,\, r\cdot k-\mathbf{c}^n_{d,k} } \\
\,\,\,=\,\,\, \mathbf{p}_r(\mathcal{N}_{d-k}(\ell)-\mathcal{N}_{d-k-1}(\ell)) - \codim(\mathcal{N}_{d-k}(\ell)-\mathcal{N}_{d-k-1}(\ell)),
\end{multline*}
so projection induces a homomorphism
\[
 p_\ast \colon I^{\mathbf{q}_{r+1}}H_\ast(\overline{\mathcal{M}}_d(\ell)) \to I^{\mathbf{p}_r}H_\ast(\mathcal{M}_d(\ell))
\]
which is seen to be the inverse of $i_\ast$ as the straight-line homotopy between $i\circ p$ and the identity on $\overline{\mathcal{M}}_d(\ell)$ induces the required chain homotopy.

The pseudomanifold $\overline{\mathcal{M}}_d(\ell)$ is non-compact, but we can form a similar compact pseudomanifold by letting
\begin{eqnarray*}
 \hat{\mathcal{M}}_d(\ell)&=& \mathcal{C}_d(\ell)\times_{\SO(d-1)}D^{d-1}\\
 \partial\hat{\mathcal{M}}_d(\ell)&=&\mathcal{C}_d(\ell)\times_{\SO(d-1)}S^{d-2}
\end{eqnarray*}
and
\begin{eqnarray*}
 \tilde{\mathcal{M}}_d(\ell)&=& \hat{\mathcal{M}}_d(\ell) / \partial\hat{\mathcal{M}}_d(\ell)
\end{eqnarray*}
with extra stratum $\ast$ corresponding to $\partial\hat{\mathcal{M}}_d(\ell)$. If we extend the perversity $\mathbf{q}_r$ by defining
\begin{eqnarray*}
 \mathbf{q}_r(\ast)&=& r(d-2),
\end{eqnarray*}
we then have
\begin{eqnarray}\label{eqn_thom}
 I^{\mathbf{q}_r}H_{\mathbf{d}^{n+1-r}_d}(\tilde{\mathcal{M}}_d(\ell)) & \cong & I^{\mathbf{q}_r}H_{\mathbf{d}^{n+1-r}_d}(\overline{\mathcal{M}}_d(\ell))
\end{eqnarray}
for $r>1$, since allowable $k$-chains in $\tilde{\mathcal{M}}_d(\ell)$ with $k\leq \mathbf{d}^{n+1-r}_d+1$ cannot intersect the extra stratum $\ast$. For $r=1$ this does not work, as the generator of $I^{\mathbf{q}_1}H_{\mathbf{d}^n_d}(\overline{\mathcal{M}}_d(\ell))$ bounds in $\tilde{\mathcal{M}}_d(\ell)$ in an allowable way. We can solve this by resorting to non-Goresky-MacPherson perversities, setting
\begin{eqnarray*}
 \mathbf{q}_1(\ast)&=&0.
\end{eqnarray*}
Now (\ref{eqn_thom}) also holds for $r=1$, and we can form the intersection ring for $\tilde{\mathcal{M}}_d(\ell)$ given by
\begin{eqnarray*}
 I\!H^\ast(\tilde{\mathcal{M}}_d(\ell))&=&\bigoplus_{r=0}^{n-d}I^{\mathbf{q}_r}H_{\mathbf{d}^{n+1-r}_d}(\tilde{\mathcal{M}}_d(\ell))
\end{eqnarray*}
with the intersection product coming from \cite[Thm.5.3]{friedg}. Note that products involving $r_1+r_2>n-d$ are considered $0$.

Define an element
\begin{eqnarray*}
 R&\in & I^{\mathbf{p}_1}H_{\mathbf{d}^{n-1}_d}(\mathcal{M}_d(\ell))
\end{eqnarray*}
as follows. The fundamental class $[\mathcal{M}_d(\ell)] \in I^{\mathbf{p}_0}H_{\mathbf{d}^n_d}(\mathcal{M}_d(\ell))$ represents a generator $X\in I^{\mathbf{q}_1}H_{\mathbf{d}^n_d}(\tilde{\mathcal{M}}_d(\ell))$ which we can intersect with itself to obtain an element $X^2\in I^{\mathbf{q}_2}H_{\mathbf{d}^n_d}(\tilde{\mathcal{M}}_d(\ell))$. We then let $R=p_\ast(X^2)$.

We want to represent $R$ more directly. For this, define
\[
 i_j\colon \mathcal{M}_d(\ell) \to \overline{\mathcal{M}}_d(\ell)
\]
for each $j\in \{1,\ldots,n-1\}$ by 
\begin{eqnarray*}
 i_j([x_1,\ldots,x_{n-1}])&=&[x_1,\ldots,x_{n-1},\pi(x_j)]
\end{eqnarray*}
where $\pi\colon S^{d-1}\to \R^{d-1}$ is projection to the last $d-1$ coordinates of $S^{d-1}\subset \R^d$. Then
\[
 i_j(\mathcal{N}_{d-k}(\ell)-\mathcal{N}_{d-k-1}(\ell))\subset \overline{\mathcal{N}}_{d-k}^{d-k-1}(\ell) -\overline{\mathcal{N}}_{d-k-1}(\ell)
\]
and it is clear that $(i_j)_\ast=i_\ast\colon I^{\mathbf{p}_r}H_\ast(\mathcal{M}_d(\ell)) \to I^{\mathbf{q}_{r+1}}H_\ast(\overline{\mathcal{M}}_d(\ell))$.

Let $\ell=(\ell_1,\ldots,\ell_n)$ be a generic, ordered length vector. For $j=1,\ldots,n-1$ recall the length vector $\ell^j_+\in \R^{n-1}$ given by
\begin{eqnarray*}
 \ell^j_+&=&(\ell_1,\ldots,\hat{\ell}_j,\ldots,\ell_{n-1},\ell_n+\ell_j).
\end{eqnarray*}
If we replace the last coordinate by $\ell_n-\ell_j$, we get a length vector that we call $\ell^j_-\in \R^{n-1}$.

As in \cite{schint}, we now get elements
\[
 [\mathcal{M}_d(\ell^1_+)],\ldots,[\mathcal{M}_d(\ell^{n-1}_+)],[\mathcal{M}_d(\ell^1_-)],\ldots, [\mathcal{M}_d(\ell^{n-1}_-)] \,\,\in\,\, I^{\mathbf{p}_1}H_{\mathbf{d}^{n-1}_d}(\mathcal{M}_d(\ell)).
\]
For the moment, there is still some abiguity about the orientations of these elements. To resolve this, note that we can think of $\mathcal{M}_d(\ell^j_\pm)$ as a subset of $\overline{\mathcal{M}}_d(\ell)$ using the the standard zero section. Then
\begin{eqnarray*}
 \mathcal{M}_d(\ell)\cap i_j(\mathcal{M}_d(\ell))&=& \{[x_1,\ldots,x_{n-1},0]\,|\,x_j\in S^0\}\\
&=& \mathcal{M}_d(\ell^j_+) \sqcup \mathcal{M}_d(\ell^j_-)
\end{eqnarray*}
Choosing an orientation of $\mathcal{M}_d(\ell)$ then induces an orientation on $\mathcal{M}_d(\ell)\cap i_j(\mathcal{M}_d(\ell))$ for all $j=1,\ldots,n-1$. In particular, we get
\begin{eqnarray}\label{formofR}
 R&=&[\mathcal{M}_d(\ell^j_+)]+[\mathcal{M}_d(\ell^j_-)]
\end{eqnarray}
for all $j=1,\ldots,n-1$. As $R=p_\ast(X^2)$ we see that $2 ([\mathcal{M}_d(\ell^j_+)] + [\mathcal{M}_d(\ell^j_-)])=0$ for even $d$ which differs slightly from \cite[Lm.7.5]{schint} because of different orientation conventions.

To simplify notation, let us write
\begin{eqnarray*}
 X_j&=&[\mathcal{M}_d(\ell^j_+)]\,\,\,=\,\,\,[\mathcal{M}_d(\ell_{\{j\}})]\,\,\,\in \,\,\, I^{\mathbf{p}_1}H_{\mathbf{d}^{n-1}_d}(\mathcal{M}_d(\ell))\\
 X_j^-&=&[\mathcal{M}_d(\ell^j_-)] \,\,\,\in \,\,\, I^{\mathbf{p}_1}H_{\mathbf{d}^{n-1}_d}(\mathcal{M}_d(\ell)).
\end{eqnarray*}

\begin{remark}
 \label{rem_stratfibbun}
The space $\overline{\mathcal{M}}_d(\ell)$ is a stratified bundle in the sense of \cite{baufer}. To see this, let $M$ be a compact smooth $G$-manifold, where $G$ is a compact Lie group. If $\mathfrak{F}$ denotes the orbit category of $G$, that is, the category with objects $G/H$ for $H$ a closed subgroup of $G$, and whose morphisms are $G$-equivariant maps $G/H\to G/H'$, then $M \to M/G$ is an $\mathfrak{F}$-stratified bundle by \cite[Example 4.6]{baufer}, see also \cite{mdavis}. Now let $V$ be a vector space and $\rho\colon G \to GL(V)$ a representation. Define the category $\mathfrak{V}$ as a subcategory of topological spaces where the objects are the quotient spaces $V/H$ for $H$ a closed subgroup of $G$. There is an obvious functor $\varphi_V\colon \mathfrak{F} \to \mathfrak{V}$, and one can check that $M\times_GV$ agrees with the coend construction $M^\circ \otimes_\mathfrak{F}\varphi_V$ described in \cite[\S 6]{baufer}. In particular, $p\colon M\times_G V\to M/G$ is a $\mathfrak{V}$-stratified fibre 
bundle in the sense of \cite{baufer}. Notice however that it is not a stratified vector bundle in general.

We can now think of the element $R$ above as an \em Euler class\em, in that it represents an obstruction for the existence of a stratified non-zero section $\sigma\colon \mathcal{M}_d(\ell) \to \overline{\mathcal{M}}_d(\ell)$. One would expect that the above constructions can extend to $M\times_G V\to M/G$, and one may ask how far this can be generalized to the setting of stratified fibre bundles.
\end{remark}

\section{The intersection ring of $\mathcal{M}_d(\ell)$}

For the next lemma, we also use the notation
\begin{eqnarray*}
 X_K&=&\prod_{i\in K} X_i
\end{eqnarray*}
for $K\subset \{1,\ldots,n-1\}$.

\begin{lemma}\label{lem_duals}
 Let $d\geq 4$ and $\ell$ a generic, ordered $d$-regular length vector, and $J\subset\{1,\ldots,n-1\}$ such that $J\cup \{n\}$ is $\ell$-short, and $K\subset \{1,\ldots,n-1\}$ with $|K|=|J|$. Then there exists $Y_J\in I^\mathbf{0}H_{|J|(d-1)}(\mathcal{M}_d(\ell))$ with
\begin{eqnarray*}
 X_K \cdot Y_J &=&\left\{ \begin{array}{cl} 1 & K=J \\ 0 & \mbox{\rm else} \end{array}\right.
\end{eqnarray*}
and
\begin{eqnarray*}
 R \cdot Y_J &=& 0.
\end{eqnarray*}

\end{lemma}

\begin{proof}
In \cite[Lm.8.1]{schint} explicit duals $Y_J$ for $X_J$ were constructed by defining appropriate embeddings of $(S^{d-1})^{|J|}$ into $\mathcal{M}_d(\ell)$. The relation $Y_J\cdot X_K=0$ for $K\not=J$ was a consequence of being able to avoid letting the $k$-th coordinate $x_k$ of the element in $\mathcal{M}_d(\ell)$ point in the same direction as $x_n$.

To do this the robot arm consisting of those links which were not part of $J\cup \{n\}$ had to trace the area in $\R^d$ that could be reached by the robot arm consisting of the links in $J$ and which started at $\ell_n e_1\in \R^d$. To do this, the first robot arm has to trace a straight line, and then reach all other points using appropriate rotations, where not all links would rotate the same way. For the first link (the one connected to the origin), one can avoid completely this latter rotation, so one just has to avoid the points $\pm e_1$ during the trace of the straight line, which can easily be done. Note that even for $d=4$ we have $n\geq 6$ to avoid trivial cases, so that the first robot arm  has at least four links.

Such a dual will then also satisfy $Y_J\cdot X_k^-=0$, and therefore
\begin{eqnarray*}
 R\cdot Y_j &=& (X_k+X_k^-)\cdot Y_j\\
&=& 0,
\end{eqnarray*}
where we use (\ref{formofR}).
\end{proof}

Let $\mathbf{p}_1'$ be the dual perversity to $\mathbf{p}_1$, that is, the perversity with $\mathbf{p}_1'+\mathbf{p}_1=\mathbf{t}$.

\begin{lemma}\label{lem_geoconst}
 Let $d\geq 5$ be odd and $\ell\in \R^n$ a generic, $d$-regular, ordered length vector with $n \geq d+3$. Then there exists an element $Y\in I^{\mathbf{p}_1'}H_{d-1}(\mathcal{M}_d(\ell))$ with
\begin{eqnarray*}
 X_{n-1} \cdot Y &=& 0\\
 X_{n-1}^- \cdot Y &=& 1.
\end{eqnarray*}

\end{lemma}

\begin{remark}
 The proof of Lemma \ref{lem_geoconst} relies on a delicate geometric construction that we postpone to Section \ref{sec_geoconst}. In the case $d=5$ this construction simplifies significantly, and we give this simplified construction here as it already contains some of the ideas required in the general case. We will assume that $\ell$ is ordered. By the equivariant Morse-Bott function on $\mathcal{C}_5(\ell)$ constructed in \cite{schuet} we get an $\SO(4)$-equivariant embedding $\mathcal{C}_5(\ell^{n-1}_-)\times D^{4}$ into $\mathcal{C}_5(\ell)$, where $\SO(4)$ acts diagonally on $\mathcal{C}_5(\ell^{n-1}_-)\times D^4$, and $\mathcal{C}_5(\ell^{n-1}_-)\times \{0\}$ corresponds to the obvious embedding $\mathcal{C}_5(\ell^{n-1}_-)\subset \mathcal{C}_5(\ell)$.

Pick an element $p\in \mathcal{C}_5(\ell^{n-1}_-)$ of rank $5$ (or $4$). Now define $f\colon S^3\to \mathcal{C}_5(\ell^{n-1}_-)$ by $f(q)=q\cdot p$ where we think of $S^3$ as a subgroup of $\SO(4)$ via quaternion multiplication. Now observe that $\mathcal{C}_5(\ell^{n-1}_-)$ is $3$-connected. Firstly, the Morse-Bott function in \cite[\S 3]{schuet} can be modified to a Morse function which has critical points only of index $4(n-3-k)$ or $4(n-3-k)+3$ for $k\in \{0,\ldots,n-3\}$, which makes $\mathcal{C}_5(\ell^{n-1}_-)$ simply connected. Furthermore, the cohomology calculation in \cite[Thm.2.1]{fahasc} shows that the first non-trivial homology group of $\mathcal{C}_5(\ell^{n-1}_-)$ has at least degree $4$, which means this space is $3$-connected. Note that we require $\ell^{n-1}_-$ to be $3$-regular, which is implied by $\ell$ being $5$-regular, to ensure the vanishing of the third homology group.

We can therefore extend $f$ to a map $F\colon D^4 \to \mathcal{C}_5(\ell^{n-1}_-)$ which can even be an embedding. Also, this embedding can be made transverse to the map $g\colon \mathcal{C}_3(\ell_-^{n-1})\times \SO(4) \to \mathcal{C}_5(\ell_-^{n-1})$ given by $g(x,A)=Ax$. For dimension reasons, this means that $F$ misses $g$, so that all $F(x)$ have rank at least $4$. Finally, the map $\tilde{F}\colon D^4 \to \mathcal{C}_5(\ell^{n-1}_-) \times D^4$ given by $\tilde{F}(x)=(F(x),x)$ induces a map $\bar{F}\colon S^4 \to \mathcal{M}_5(\ell)$ with $\bar{F}(S^4) \cap \mathcal{M}_5(\ell^{n-1}_-)=\{[p]\}$. By letting $F$ be constant in a neighborhood of $0$ (which lets $F$ no longer be an embedding, but lets $\tilde{F}$ remain an embedding) this intersection is transverse. Furthermore, $\bar{F}(S^4) \cap \mathcal{M}_5(\ell^{n-1}_-)=\emptyset$. Therefore $\bar{F}(S^4)$ represents the required element $Y$, and since all points in $\bar{F}(S^4)$ have rank at least $4$, we even get an element $Y\in I^\mathbf{0}H_4(\mathcal{M}_d(\ell))$. Also, note that we only require $n\geq d+2=7$ here, if $n=6$, the element $X_{n-1}^-=0$.

A similar construction can be done in the case $d=9$ using octonian multiplication, however, it is not clear how this construction could generalize to the other cases of odd $d$.
\end{remark}

\begin{proposition}
\label{prop_relations}
 Let $d\geq 5$ be odd and $\ell\in \R^n$ be a $d$-regular, ordered, generic length vector with $n\geq d+2$. Let $k=a_1(\ell)$. Then the intersection ring $I\!H^\ast(\mathcal{M}_d(\ell))$ is generated by elements $R, X_1, \ldots, X_k$ which satisfy the following relations:
\begin{enumerate}
 \item $RX_i = X_i^2$ for all $i=1,\ldots,k$.
 \item $X_{i_1}\cdots X_{i_m}$ if $\{i_1,\ldots,i_m,n\}$ is $\ell$-long.
\end{enumerate}
For $n\geq d+3$ we can choose $R$ to be the Euler class of $\ell$.
\end{proposition}

This is not a complete list of relations, for example we have $R^m=0$ for $m$ large enough simply by the construction of the intersection ring. Notice also that for $n=d+2$ we cannot have non-trivial products for degree reasons.

\begin{proof}
 Let $n=d+2$. Then we just choose the elements $R, X_1, \ldots, X_k$ so that they form a basis of $I^{\mathbf{p}_1}H_{\mathbf{d}^{n-1}_d}(\mathcal{M}_d(\ell))$, compare Proposition \ref{odd_generators}. Any products among these elements are zero for degree reasons, so the relations are trivially satisfied.

Now let $n>d+2$. We now choose $R$ and $X_i$ as in Section \ref{sec_bundles}. By Lemma \ref{lem_duals} and Lemma \ref{lem_geoconst} these elements are linearly independent and form a basis because of Proposition \ref{odd_generators}. By Section \ref{sec_bundles}(2) we get $R=X_j+X^-_j$ for all $j=1,\ldots,n-1$, so $R X_i=X_iX_i+X_i^-X_i$. Now $X_i^-X_i$ is represented by $\mathcal{M}_d(\ell^i_-) \cap \mathcal{M}_d(\ell^i_+)=\emptyset$, so (1) follows. Also,  $X_{i_1}\cdots X_{i_m}$ is represented by $\mathcal{M}_d(\ell^{i_1}_+) \cap \cdots \cap \mathcal{M}_d(\ell_+^{i_m})$ which is empty by the condition that $\{i_1,\ldots,i_m,n\}$ is $\ell$-long. Therefore (2) holds.
\end{proof}

\begin{remark}
 We want to compare the previous result to the cohomology ring of $\mathcal{M}_3(\ell)$ determined in \cite{hauknu}. For a generic, ordered length vector $\ell\in \R^n$ their Theorem 6.4 states that
\begin{eqnarray*}
 H^\ast(\mathcal{M}_3(\ell)) &\cong & \Z[R,V_1,\ldots, V_{m-1}]/I_\ell
\end{eqnarray*}
where $R$ and $V_i$ are of degree 2 and $I_\ell$ is the ideal generated by the three families
\begin{eqnarray*}
 V_i^2+RV_i & & i=1,\ldots,n-1\\
 \prod_{i\in L} V_i & & L\subset \{1,\ldots,n-1\} \mbox{ with }L\cup\{n\}\, \ell\mbox{-long}\\
 \sum_{S\subset L,S\in \mathcal{S}_\ast(\ell)} \left( \prod_{i\in S}V_i\right)R^{|L-S|-1}& & L\subset \{1,\ldots,n-1\} \, \ell\mbox{-long}
\end{eqnarray*}
The first two families correspond to the relations in Proposition \ref{prop_relations} after a change of sign. The third family is more complicated, and we will not try to find the corresponding relations for the intersection ring. We note however the following: If $L\subset \{1,\ldots,n-1\}$ is $\ell$-long and $S\subset L$ is $\ell$-short, then either $|L-S|>1$ or $S\cup \{n\}$ is $\ell$-long. Therefore the relations in the third family are of the form $RW$ with $W\in \Z[R,V_1,\ldots, V_{m-1}]$. This was already observed in \cite[Lm.5]{fahasw}.

In the case $d=3$ the stratified bundle $p\colon \overline{\mathcal{M}}_3(\ell)\to \mathcal{M}_3(\ell)$ can be viewed as a complex line bundle, and it is shown in \cite[Prop.7.3]{hauknu} that the negative Chern class agrees with $R$. Note that our $R$ would correspond to the positive Chern (Euler) class, which is consistent with the change of sign in the first relation of the next lemma below.
\end{remark}

\begin{lemma}
\label{lem_relations}
 Let $d\geq 5$ be odd and $\ell\in \R^n$ be a $d$-regular, ordered, generic length vector with $n\geq d+2$ and let $k=a_1(\ell)$. Let $I_\ell$ be the kernel of the surjection $\Phi\colon \Z[R,X_1,\ldots,X_k] \to I\!H^\ast(\mathcal{M}_d(\ell))$ induced by Proposition \ref{prop_relations}. Then there exist elements $W_1,\ldots,W_l\in\Z[R,X_1,\ldots,X_k]$ for some $l\geq 1$ so that $I_\ell$ is generated by relations
\begin{enumerate}
 \item $RX_i - X_i^2$ for all $i=1,\ldots,k$.
 \item $X_{i_1}\cdots X_{i_m}$ if $\{i_1,\ldots,i_m,n\}$ is $\ell$-long.
 \item $RW_i$ for $i=1,\ldots,l$.
\end{enumerate}
\end{lemma}

\begin{proof}
By the Hilbert Basis Theorem we know that $I_\ell$ is finitely generated. Since the elements of the form (1) and (2) are in $I_\ell$ by Proposition \ref{prop_relations}, we can simply add them to any finite generating set. So these elements together with finitely many elements $V_1,\ldots, V_l\in \Z[R,X_1,\ldots,X_k]$ form a generating set. Let us write each $V_i$ as a linear combination of monomials $V_i=\sum_{j=1}^{j_i} a_{ij}V_{ij}$ with $a_{ij}\in \Z-\{0\}$. We can first assume that no monomial $V_{ij}$ contains more than one factor of any $X_v$, for we could replace this monomial with the corresponding monomial having $R^{u-1}X_v$ in place of $X_v^u$ using a relation from (1). Now if a monomial $V_{ij}$ has no factor $R$, we can write it as $V_{ij}=X_{u_1}\cdots X_{u_v}$ with $J=\{u_1,\ldots,u_v\}$ and $|J|=v$. If $J\cup \{n\}$ is $\ell$-long, we can remove $V_{ij}$ using a relation of the form (2). If $J\cup \{n\}$ is $\ell$-short we get a Poincar\'e dual $Y_J$ to $V_{ij}$ from Lemma \ref{lem_duals} 
with $R\cdot Y_J=0$. Then $V_i\cdot Y_J=a_{ij}\not=0$ which contradicts $V_i\in I_\ell$. Therefore such a monomial cannot appear in $V_i$. It follows that $V_i=RW_i$ for some $W_i\in \Z[R,X_1,\ldots,X_k]$.
\end{proof}

Let us now consider $\Z/2$ coefficients. To simplify our discussion, we will simply tensor the integral intersection ring with $\Z/2$ to obtain a new ring that we denote by
\begin{eqnarray*}
 I\!H_{\Z/2}^\ast(\mathcal{M}_d(\ell))&=& I\!H^\ast(\mathcal{M}_d(\ell))\otimes \Z/2.
\end{eqnarray*}

\begin{corollary}
\label{cor_thefacering}
 Let $d\geq 5$ be odd and $\ell\in \R^n$ be a $d$-regular, generic length vector with $d\geq d+3$. Let $R$ be the Euler class of $\ell$ with $\Z/2$ coefficients. Then
\begin{eqnarray*}
 I\!H^\ast_{\Z/2}(\mathcal{M}_d(\ell))/\langle R\rangle &\cong & \Lambda_{\Z/2}[\mathcal{S}_\ast(\ell)].
\end{eqnarray*}

\end{corollary}

\begin{proof}
 We have the relations from Proposition \ref{prop_relations} which reduce to $X_i^2$ for all $i=1,\ldots,k$ and $X_{i_1}\cdots X_{i_m}$ if $\{i_1,\ldots,i_m,n\}$ is $\ell$-long. Over $\Z/2$ this reduces to the exterior face ring of the short subsets.
\end{proof}

We now need to find a way to detect $R$ in terms of intersection products. This is similar to the argument used in \cite{fahasw}, however, since we have worse information about the intersection ring, the argument is a bit more involved. As a start, we need the following result.

\begin{lemma}
\label{lem_direct_sum}
 Let $d\geq 5$ be odd and $\ell\in \R^n$ be a $d$-regular, generic length vector with $n\geq d+3$ and $\mathcal{M}_d(\ell)\not=\emptyset$. Let $i,j\in \{1,\ldots,n-1\}$ with $i<j$ and let $\ell^{ij}_{+-}=(\ell^j_-)^i_+\in \R^{n-2}$. Then $I^\mathbf{0}H_{\mathbf{d}^{n-2}_d}(\mathcal{M}_d(\ell^{ij}_{+-}))$ is a direct summand of $I^{\mathbf{p}_1}H_{\mathbf{d}^{n-2}_d}(\mathcal{M}_d(\ell^j_-))$ and $I^{\mathbf{p}_1}H_{\mathbf{d}^{n-2}_d}(\mathcal{M}_d(\ell^j_-))$ is a direct summand of  $I^{\mathbf{p}_2}H_{\mathbf{d}^{n-2}_d}(\mathcal{M}_d(\ell))$.
\end{lemma}

\begin{proof}
Both statements have essentially the same proof, we will therefore focus on the second statement. Note that inclusion $\mathcal{M}_d(\ell^j_-)\subset \mathcal{M}_d(\ell)$ induces a map $I^{\mathbf{p}_1}H_{\mathbf{d}^{n-2}_d}(\mathcal{M}_d(\ell^j_-))\to I^{\mathbf{p}_2}H_{\mathbf{d}^{n-2}_d}(\mathcal{M}_d(\ell))$. To see that this map is split-injective, we use the Morse argument used in \cite[\S 7]{schint}. There is an $\SO(d-1)$-invariant Morse-Bott function on $\mathcal{C}_d(\ell)\to \R$ whose absolute minimum is $\mathcal{M}_d(\ell^j_-)$ and whose other critical manifolds are spheres of dimension $d-2$ of index $k(d-1)$ for $k\in \{1,\ldots,n-3\}$. This gives rise to a filtration of $\mathcal{M}_d(\ell)$
\[
 \emptyset \subset \mathcal{M}^0 \subset \mathcal{M}^1 \subset\cdots \subset \mathcal{M}^m=\mathcal{M}_d(\ell)
\]
where $\mathcal{M}^0$ contains $\mathcal{M}_d(\ell^j_-)$ as a deformation retract in a stratification preserving way so that $I^{\mathbf{p}_1}H_{\mathbf{d}^{n-2}_d}(\mathcal{M}_d(\ell^j_-)) \cong I^{\mathbf{p}_2}H_{\mathbf{d}^{n-2}_d}(\mathcal{M}^0)$. We need to show that $I^{\mathbf{p}_2}H_{\mathbf{d}^{n-2}_d}(\mathcal{M}^{l-1})\to I^{\mathbf{p}_2}H_{\mathbf{d}^{n-2}_d}(\mathcal{M}^l)$ is split-injective for all $l=1,\ldots,m$. We have the long exact sequence
\begin{multline*}
 \cdots \longrightarrow I^{\mathbf{p}_2}H_{\mathbf{d}^{n-2}_d+1}(\mathcal{M}^l,\mathcal{M}^{l-1}) \longrightarrow I^{\mathbf{p}_2}H_{\mathbf{d}^{n-2}_d}(\mathcal{M}^{l-1})\longrightarrow \\ I^{\mathbf{p}_2}H_{\mathbf{d}^{n-2}_d}(\mathcal{M}^l)\longrightarrow  I^{\mathbf{p}_2}H_{\mathbf{d}^{n-2}_d}(\mathcal{M}^l,\mathcal{M}^{l-1}) \longrightarrow \cdots 
\end{multline*}
The proof is going to be along the lines of the proof of \cite[Lm.7.1]{schint}.

Recall notation from \cite[\S 6]{schint}, namely for non-negative integers $m,k$ with $m\geq k$ let
\begin{eqnarray*}
 \mathcal{N}^{m,k} &=& ((D^{d-1})^k\times (D^{d-1})^{m-k})/\SO(d-2),\\
 \partial_-\mathcal{N}^{m,k} &=& (\partial((D^{d-1})^k)\times (D^{d-1})^{m-k})/\SO(d-2).
\end{eqnarray*}
Then $I^{\mathbf{p}_2}H_r(\mathcal{M}^l,\mathcal{M}^{l-1})\cong I^{\mathbf{p}_2}H_r(\mathcal{N}^{n-3,k_l},\partial_-\mathcal{N}^{n-3,k_l})$, where $k_l$ is the index of the critical point contained in $\mathcal{M}^l-\mathcal{M}^{l-1}$, that is, the corresponding critical sphere $S^{d-2}$ is of index $k_l(d-1)$.

Assume that $k_l\leq n-5$. Then
\begin{eqnarray*}
 I^{\mathbf{p}_2}H_r(\mathcal{N}^{n-3,k_l},\partial_-\mathcal{N}^{n-3,k_l}) &\cong & I^\mathbf{0}H_r(\mathcal{N}^{n-5,k_l},\partial_-\mathcal{N}^{n-5,k_l})
\end{eqnarray*}
by \cite[Lm.6.1]{schint}. For $k_l=n-5$ this has only one non-vanishing group in degree $r=\mathbf{d}^{n-2}_d$ by \cite[\S 5]{schint}, which is $\Z$. For $k_l< n-5$ we can use Poincar\'e duality (taking torsion into account, using \cite[Cor.4.4.3]{friedg})
\begin{eqnarray*}
 I^\mathbf{0}H_r(\mathcal{N}^{n-5,k_l},\partial_-\mathcal{N}^{n-5,k_l}) &\cong & I^\mathbf{t}H_{\mathbf{d}^{n-2}_d-r}(\mathcal{N}^{n-5,k_l},\partial_+\mathcal{N}^{n-5,k_l})
\end{eqnarray*}
and the latter is just ordinary homology, which vanishes for $r=\mathbf{d}^{n-2}_d,\mathbf{d}^{n-2}_d+1$ as $\partial_+\mathcal{N}^{n-5,k_l}\not=\emptyset$ for $k_l<n-5$. Therefore the map  $I^{\mathbf{p}_2}H_{\mathbf{d}^{n-2}_d}(\mathcal{M}^{l-1})\to I^{\mathbf{p}_2}H_{\mathbf{d}^{n-2}_d}(\mathcal{M}^l)$ is split-injective if $k_l\leq n-5$.

If $k_l=n-4$, we get
\begin{eqnarray*}
 I^{\mathbf{p}_2}H_r(\mathcal{N}^{n-3,n-4},\partial_-\mathcal{N}^{n-3,n-4}) &\cong & I^{\mathbf{p}_1}H_r(\mathcal{N}^{n-4,n-4},\partial_-\mathcal{N}^{n-4,n-4})
\end{eqnarray*}
and the latter is $\Z/2$ for $r=\mathbf{d}^{n-2}_d+1$ and $0$ for $r=\mathbf{d}^{n-2}_d$ by \cite[Prop.6.3]{schint}. Therefore our map is again split-injective.

To calculate $I^{\mathbf{p}_2}H_r(\mathcal{N}^{n-3,n-3},\partial_-\mathcal{N}^{n-3,n-3})$ a cellular chain complex is identified in \cite[\S 6]{schint}, and the lowest dimensional cell is of dimension $\mathbf{d}^{n-2}_d+2$, compare \cite[Lm.6.2]{schint}. Therefore the homology groups in degree $\mathbf{d}^{n-2}_d$ are not affected, and the result follows.
\end{proof}

\begin{corollary}
 \label{cor_non-zero}
Let $d\geq 5$ be odd and $\ell\in \R^n$ be a $d$-regular, generic length vector with $n\geq d+3$. Let $i,j\in \{1,\ldots,n-1\}$ with $i<j$. If $X_i$ is a non-zero element of $I^{\mathbf{p}_1}H_{\mathbf{d}^{n-1}_d}(\mathcal{M}_d(\ell))$, then $X_iX_j^-$ is a non-zero element of $I^{\mathbf{p}_2}H_{\mathbf{d}^{n-2}_d}(\mathcal{M}_d(\ell))\otimes \Z/2$.
\end{corollary}

\begin{proof}
Observe that $X_iX_j^-$ is represented by $\mathcal{M}_d(\ell^i_+)\cap \mathcal{M}_d(\ell^j_-)$, and since $X_i\not=0$ implies $\{i,n\}$ is $\ell$-short, this set is non-empty. Therefore $X_iX_j^-$ is the image of the fundamental class in $I^\mathbf{0}H_{\mathbf{d}^{n-2}_d}(\mathcal{M}_d(\ell^i_+)\cap \mathcal{M}_d(\ell^j_-))$. By Lemma \ref{lem_direct_sum} it follows that $0\not= X_iX_j^- \in I^{\mathbf{p}_2}H_{\mathbf{d}^{n-2}_d}(\mathcal{M}_d(\ell))\otimes \Z/2$.
\end{proof}

\begin{proof}[Proof of Theorem \ref{main_theorem}]
 The cases $n=2,3$ were covered in \cite{fahasw, schuew}, while $d\geq 4$ even is in \cite{schint}. It remains to consider $d\geq 5$ odd. That $\ell$ is $d$-regular implies $n\geq d+1$. If $n=d+1$, there is only one chamber up to permutation, so there is nothing to show. Similarly, if $n=d+2$, then $a_2(\ell)=0$ and the chamber is determined by $a_1(\ell)$, which is recovered from $I^{\mathbf{p}_1}H_{\mathbf{d}^{n-1}_d}(\mathcal{M}_d(\ell))$, a homeomorphism invariant.

We can therefore assume that $n\geq d+3$. Assume also that $\ell$ is ordered. We want to say that $R$ is the unique element of $I^{\mathbf{p}_1}H_{\mathbf{d}^{n-1}_d}(\mathcal{M}_d(\ell))\otimes \Z/2$ such that multiplication by $R$ induces the squaring homomorphism $$\mathrm{Sq}\colon I^{\mathbf{p}_1}H_{\mathbf{d}^{n-1}_d}(\mathcal{M}_d(\ell))\otimes \Z/2 \to I^{\mathbf{p}_2}H_{\mathbf{d}^{n-2}_d}(\mathcal{M}_d(\ell))\otimes \Z/2.$$
By Proposition \ref{prop_relations}, $R$ certainly has this property. To get uniqueness, let us also assume that $a_2(\ell)>0$, so that there exist $i\not=j$ with $X_iX_j\not=0$ (we can use $i=1$, $j=2$ since $\ell$ is ordered). Now let
\begin{eqnarray*}
 R'&=& \varepsilon R + X_{i_1}+\cdots + X_{i_u} \,\,\, \in \,\,\, I^{\mathbf{p}_1}H_{\mathbf{d}^{n-1}_d}(\mathcal{M}_d(\ell))\otimes \Z/2
\end{eqnarray*}
satisfy $R'X=X^2$ for all $X\in I^{\mathbf{p}_1}H_{\mathbf{d}^{n-1}_d}(\mathcal{M}_d(\ell))\otimes \Z/2$. In particular
\begin{eqnarray*}
 X_j^2&=& \varepsilon X_j^2 + X_{i_1}X_j+\cdots + X_{i_u}X_j
\end{eqnarray*}
for all $j=1,\ldots,n-1$. If any $X_{i_v}X_j$ were non-zero for $j\not=i_v$, then $X_j^2 \cdot Y_{i_v,j}\not=0$ for the dual $Y_{i_v,j}$ from Lemma \ref{lem_duals}. But since $X_j^2\cdot Y_{i_v,j}= X_j\cdot R \cdot Y_{i_v,j} =0$, we get $X_{i_v}X_j=0$ for all $i_v\not= j$. In particular $i_v\not=1$ or $2$ by the assumption $X_1X_2\not=0$. By using $j=i_1$ (assuming that $u\geq 1$) we also get
\begin{eqnarray*}
 X_{i_v}^2 &=& \varepsilon X_{i_v}^2 + X_{i_v}^2
\end{eqnarray*}
by multiplication with $R'$, but we also have
\begin{eqnarray*}
 X_{i_v}^2 &=& (X_1+X_1^-) X_{i_v}\\
&=& X_1^- X_{i_v}
\end{eqnarray*}
by multiplication with $R=X_1+X_1^-$. Therefore $X_{i_v}^2\not=0$ by Lemma \ref{cor_non-zero} which means $\varepsilon=0\in \Z/2$. Now
\begin{eqnarray*}
 X_1^2 &=& (X_{i_1}+\cdots + X_{i_u}) X_1\\
&=&0
\end{eqnarray*}
contradicting $X_1^2=(X_{i_u}+X_{i_u}^-)X_1=X_{i_u}^-\cdot X_1\not=0$.

So under the extra condition that $a_2(\ell)>0$ we get that $R$ is the only element in $I^{\mathbf{p}_1}H_{\mathbf{d}^{n-1}_d}(\mathcal{M}_d(\ell))\otimes \Z/2$ such that multiplication by $R$ gives $\mathrm{Sq}$.

Let $\ell,\ell'\in \R^n$ be generic, $d$-regular length vectors with $\mathcal{M}_d(\ell)$ homeomorphic to $\mathcal{M}_d(\ell')$. By \cite{gorma2} there is an isomorphism of intersection rings $I\!H^\ast(\mathcal{M}_d(\ell))\cong I\!H^\ast(\mathcal{M}_d(\ell'))$. If the Euler class $R_\ell$ of $\ell$ would not be unique with the squaring property (after tensoring with $\Z/2$), then neither would be $R_{\ell'}$ and we would get $a_2(\ell)=0=a_2(\ell')$. But, up to permutation, the chamber of any $\ell$ with $a_2(\ell)>0$ is determined by $a_1(\ell)$ which we can obtain from the dimension of $I^{\mathbf{p}_1}H_{\mathbf{d}^{n-1}_d}(\mathcal{M}_\ell)\otimes \Z/2$. So $\ell$ and $\ell'$ are in the same chamber up to permutation.

We can therefore assume that both $R_\ell$ and $R_{\ell'}$ are unique with the squaring property, and the isomorphism of intersection rings induces an isomorphism of exterior face rings $\Lambda_{\Z/2}[\mathcal{S}_\ast(\ell)] \cong \Lambda_{\Z/2}[\mathcal{S}_\ast(\ell')]$ by Corollary \ref{cor_thefacering}, which induces an isomorphism of simplicial complexes $\mathcal{S}_\ast(\ell)\cong \mathcal{S}_\ast(\ell')$ by \cite{gubela}. This implies that $\ell$ and $\ell'$ are in the same chamber up to permutation as in \cite{fahasw}.
\end{proof}

\section{Proof of Lemma \ref{lem_geoconst}}
\label{sec_geoconst}
Let us begin with the strategy of the proof. From the equivariant Morse-Bott function in \cite{schuet} we get an equivariant neighborhood of $\mathcal{C}_d(\ell^{n-1}_-)$ in $\mathcal{C}_d(\ell)$ of the form $\mathcal{C}_d(\ell^{n-1}_-)\times D^{d-1}$ where $\SO(d-1)$ acts diagonally on the factors. As $d\geq 5$ is odd, there is $k\geq 2$ with $d-1=2k$, and we can write $\R^d=\R\times \C^k$. We want to construct an $S^1$-equivariant map $f\colon S^{d-2} \to \mathcal{C}_d(\ell^{n-1}_-)$ which extends (non-equivariantly) to $f\colon D^{d-1} \to \mathcal{C}_d(\ell^{n-1}_-)$ and which is constant near $0\in D^{d-1}$ and so that $f(0)$ has rank at least $d-1$. Then $f$ induces a map from complex projective space $F\colon \CP^k \to \mathcal{M}_d(\ell)$ via $F(z)=q(f(z),z)$, where $q\colon \mathcal{C}_d(\ell^{n-1}_-)\times D^{d-1} \to \mathcal{M}_d(\ell)$ is inclusion followed by the quotient map. Since $f$ is constant near $0$ and has rank $d-1$ we get that $F$ is an embedding near the point corresponding to $0$, and $F(\CP^k)$ intersects $\mathcal{M}_d(\ell^{n-1}_-)$ transversely in exactly one point, while $F(\CP^k) \cap \mathcal{M}_d(\ell^{n-1}_+)=\emptyset$. The required element is then $Y=F_\ast[\CP^k]$.

To get the $S^1$-equivariant map $f$ we will actually construct an $(S^1)^k$-equivariant map, where each factor $S^1$ acts on its respective coordinate in $\C^k$. As a result, the map $F$ will hit several singular strata in $\mathcal{M}_d(\ell)$. The next lemma gives a criterion so that $F$ induces an element of $I^{\mathbf{p}_1'}H_{d-1}(\mathcal{M}_d(\ell))$.

\begin{lemma}\label{lem_strata}
 Let $F\colon \CP^k \to \mathcal{M}_d(\ell)$ be a stratum-preserving map for some stratification of $\CP^k$ that has only even-dimensional strata. Assume that for $l=0,\ldots, k-2$ the strata contained in $F^{-1}(\mathcal{N}_{3+2l}(\ell)-\mathcal{N}_{2+2l}(\ell))$ have dimension at most $2l$, and that $F^{-1}(\mathcal{N}_{2+2l}(\ell)-\mathcal{N}_{1+2l}(\ell))=\emptyset$. Then $F$ induces a homomorphism
\[
 F_\ast \colon I^\mathbf{0}H_{d-1}(\CP^k) \to I^{\mathbf{p}_1'}H_{d-1}(\mathcal{M}_d(\ell)).
\]
\end{lemma}

\begin{proof}
Note that the condition of the Lemma states that a stratum $S$ which is send to $\mathcal{N}_{d-2(k-l-1)}(\ell)-\mathcal{N}_{d-1-2(k-l-1)}(\ell)$ has codimension at least $2(k-l)$. As 
\begin{eqnarray*}
\mathbf{p}_1'(\mathcal{N}_{d-2(k-l-1)}(\ell)-\mathcal{N}_{d-1-2(k-l-1)}(\ell))&=&\mathbf{c}^n_{d,2(k-l-1)}-2(k-l),
\end{eqnarray*}
we see that \cite[Lm.5.2]{schint} applies.
\end{proof}

Let $\Delta^{k-1}=\{(t_1,\ldots,t_k)\in \R^k\,|\, t_i\in [0,1], \, t_1+\cdots t_k=1\}$ be the standard $(k-1)$-simplex. Then $\Delta^{k-1}\cong S^{d-2}/(S^1)^k$ and we can split the quotient map via $j\colon \Delta^{k-1}\to S^{d-2}$ given by $j(t_1,\ldots,t_k)=(\sqrt{t_1},\ldots,\sqrt{t_k})\in \C^k$.

So in order to get an $(S^1)^k$-equivariant map $f\colon S^{d-2}\to \mathcal{C}_d(\ell^{n-1}_-)$ it is enough to define $f$ on $\Delta^{k-1}$. Let us start with a point in $x\in \mathcal{C}_d(\ell^{n-1}_-)$. Then $x=(x_1,\ldots,x_n)$ with $x_n=e_0$, $x_{n-1}=-e_0$, where we think of $x_j\in \R^d=\R\times \C^k$ and we think of the $\R$-coordinate as the $0$-th coordinate.

Since $\ell$ is $d$-regular and ordered, we get that $\{n-(d-1),\ldots,n-1\}$ is $\ell$-short. Let $K\subset \{1,\ldots,n-d\}$ be such that $K\cup \{n-(d-1),\ldots,n-1\}$ is long, while removing any element of $K$ would make it short.

Define $(x_1,\ldots,x_n)\in \mathcal{C}_d(\ell^{n-1}_-)$ as follows. If $j\notin K\cup \{n-(d-1),\ldots,n-1\}$, let $x_j=e_0$. If $j\in K \cup \{n-(d-1),\ldots,n-4,n-1\}$, let $x_j=e_0$. Also, let $x_{n-3}=(x_{n-3,0},x_{n-3,1},0,\ldots,0)$, $x_{n-2}=(x_{n-2,0},x_{n-2,1},0,\ldots,0)$ with $x_{n-3,0},x_{n-2,0}\in (-1,0) \subset \R$ and $x_{n-3,1}, x_{n-2,1}\in \C$ be imaginary so that $(x_1,\ldots,x_n)\in \mathcal{C}_d(\ell^{n-1}_-)$. Note that these values can be chosen by the way $K$ was defined, compare Figure \ref{fig_snake_rot}.
\begin{figure}[ht]
\begin{center}
\includegraphics[height=2cm,width=6cm]{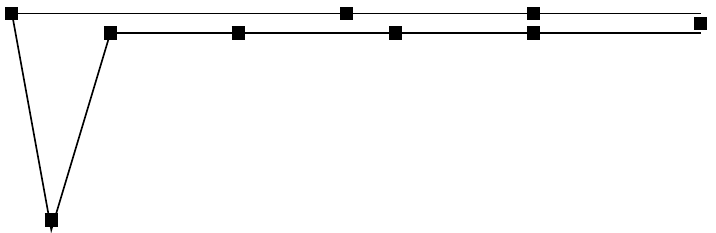}
\caption{\label{fig_snake_rot}The configuration $x$}
\end{center}
\end{figure}
The values $x_{n-3,1}, x_{n-2,1}$ are chosen imaginary, so that $e_0,e_1,x_{n-2}$ have rank $3$. Note that $e_1\in S^{d-2} \subset \C^k$ corresponds to $x_{n-1}$, once $\mathcal{C}_d(\ell^{n-1}_-)\times D^{d-1}$ is embedded in $\mathcal{C}_d(\ell)$.

We start defining $f$ by setting $f(1,0,\ldots,0)=(x_0,\ldots,x_n)$. To extend $f$, for $j=2,\ldots k-1$ and $t\in \Delta^{k-1}$ let\footnote{The letter $i$ now stands for the complex unit $i\in \C$, it will no longer be used as an index.}
\begin{eqnarray*}
 x_{n-2j}(t)&=&(-\sqrt{1-(t_j\varepsilon)^2-(t_k\varepsilon)^2},0,\ldots,0,t_j\varepsilon,0\,\ldots,0,t_k\varepsilon)\\
 x_{n-2j-1}(t)&=& (-\sqrt{1-(t_j\varepsilon)^2-(t_k\varepsilon)^2},0,\ldots,0,it_j\varepsilon,0\,\ldots,0,it_k\varepsilon)\\
 x_{n-(d-1)}(t)&=& (-\sqrt{1-(t_k\varepsilon)^2},0,\ldots,0,it_k\varepsilon)
\end{eqnarray*}
where $\varepsilon>0$ is small, and the non-zero entry is in the $j$-th complex coordinate of $\R\times \C^k$.

If $\varepsilon$ is small enough, we can define $x_{n-3}(t)\in S^1\times \{0\}\subset S^{d-1}$ and $x_{n-2}(t)\in S^{d-1}$ so that
\[
 (x_1,\ldots, x_{n-d},x_{n-(d-1)}(t),\ldots, x_{n-4}(t),x_{n-3}(t),x_{n-2}(t),x_{n-1},x_n)\in \mathcal{C}_d(\ell^{n-1}_-)
\]
with $t\in \Delta^{k-1}$. Note that $x_{n-2j}$ and $x_{n-2j-1}$ depend on $t_j$ for $j=2,\ldots,k-1$ and $t_k$. Basically, we let $x_{n-2}$ compensate for the non-zero entries in the other complex variables, thus requiring that the absolute values of the $0$-th and first coordinate in $x_{n-3}$ and $x_{n-2}$ will be slightly less. We can make this depend continuously on $t$, thus giving us a map $\Delta^{k-1}\to \mathcal{C}_d(\ell^{n-1}_-)$.

We still need to slightly change this map. As $j(\Delta^{k-1})$ is a contractible subset of $S^{d-2}$, there is a map $A\colon \Delta^{k-1} \to U(k)$ with $A(t)\cdot e_1=j(t)$. Let $f_{n-3}(t)=A(t)\cdot x_{n-3}(t)$, and $f_{n-j}(t)=x_{n-j}(t)$ for $j=4,\ldots, n-1$. Choose $f_{n-2}(t)$ so that $(f_1,\ldots,f_{n-2},-e_0,e_0)\in \mathcal{C}_d(\ell^{n-1}_-)$, that is, we have a map $f\colon \Delta^{k-1}\to \mathcal{C}_d(\ell^{n-1}_-)$.

In the next lemma we think of $(f(t),j(t))$ as an element of $\mathcal{C}_d(\ell)$ via the embedding $\mathcal{C}_d(\ell^{n-1}_-)\times D^{d-1}$. We also set $n(t)$ to be the number of coordinates $j$ with $t_j=0$ for any $t\in \Delta^{k-1}$.

\begin{lemma}\label{lem_ranks}
 We have $\mathrm{rank}(f(t),j(t))= d-1$ if $n(t)=0$, and $\mathrm{rank}(f(t),j(t))= d-2 n(t)$, if $n(t)\geq 1$. Furthermore, if $t=(t_1,\ldots,t_k)$ satisfies $t_j=0$ for $j\in \{1,\ldots,k\}$, then the $j$-th complex coordinate of each $f_m(t)$ is $0$, $m=1,\ldots,n$.
\end{lemma}

\begin{proof}
Let us write $(f(t),j(t))=(f_1,\ldots,f_{n-2},f_{n-1})\in \mathcal{C}_d(\ell)$, where$f_{n-1}$ only depends on $j(t)$. We ignore the $n$-th coordinate, as it is $e_0\in S^0$.

Since $f_j\in S^0$ for $j\leq n-d$ and $j\geq n-1$, and $f_{n-2}$ is a linear combination of the other elements, the rank can be at most $d-1$. Now assume that $t_j=0$ for $j\in \{1,\ldots,k\}$. None of the elements $f_j$ has a non-zero entry in the $j$-th complex coordinate, so again the rank can be at most $ d-2 n(t)$.

Let $n(k)=0$. Then $f_{n-1},f_{n-3}$ are the only elements (apart from $f_{n-2}$) with non-zero entries in the first complex coordinate, and these two elements are linearly independent, as $f_{n-1}$ is only in the real part, while $f_{n-3}$ is only in the imaginary part. Apart from $f_{n-1},f_{n-2},f_{n-3}$, $f_{n-4},f_{n-5}$ are the only remaining pieces with non-zero entry in the second complex coordinate. Continuing, we see that $f_{n-2j},f_{n-2j-1}$ increase the rank by two for each $j=2,\ldots,k-1$. Finally $f_{n-(d-1)}$ and $x_n$ increase the rank to $d-1$, as required.

If $n(k)>0$ let us distinguish the cases $t_1=0$ and $t_1>0$. If $t_1>0$, then $f_{n-1},f_{n-3}$ are the only elements with non-zero entries in the first complex coordinate. If $t_k=0$, then each $j\in \{2,\ldots,k-1\}$ with $t_j>0$ produces two elements $f_{n-2j},f_{n-2j-1}$ with only the $j$-th complex coordinate non-zero. Together with the elements in $S^0$, we see the rank is $d-2 n(t)$.

If $t_1>0$ and $t_k>0$, there is a $j\in \{2,\ldots,k-1\}$ with $t_j=0$. Then $f_{n-2j},f_{n-2j-1}$ have only something non-zero in the $k$-th complex coordinate, so they increase the rank by $2$. All other $j\in \{2,\ldots,k-1\}$ with $t_j>0$ have $f_{n-2j},f_{n-2j-1}$ as the only remaining elements with $j$-th complex coordinate non-zero, so they also increase the rank by $2$. Together with the elements in $S^0$, the rank is again $d-2 n(t)$.
\end{proof}

We thus have a map $\bar{f}\colon \Delta^{k-1} \to \mathcal{C}_d(\ell)$ given by $\bar{f}(t)=(f(t),j(t))$, which we can extend to an $(S^1)^k$-equivariant map $\tilde{f}\colon (S^1)^k\times \Delta^{k-1} \to \mathcal{C}_d(\ell)$, where $(S^1)^k$ acts on $\C^k$ (on both sides) by coordinate-wise multiplication. Furthermore, there is a surjection $p\colon (S^1)^k\times \Delta^{k-1} \to S^{d-2}$, which is also $(S^1)^k$-equivariant, and which induces the required $(S^1)^k$-equivariant map $$f\colon S^{d-2} \to \mathcal{C}_d(\ell)$$ by Lemma \ref{lem_ranks}.

For every subset $L\subset \{1,\ldots,k\}$ we get a subsphere $S_L\subset S^{d-2}$ of dimension $2|L|-1$, whose entries are those points $(z_1,\ldots,z_k)\in S^{d-2}$ with $z_j=0$ for $j\notin L$. It follows from Lemma \ref{lem_ranks} that $f$ followed by the quotient map from $\mathcal{C}_d(\ell)$ to $\mathcal{M}_d(\ell)$ is a stratified map, if we stratify $S^{d-2}$ by the $S_L$.

We need to extend $f$ to $D^{d-1}$. The basic idea is to stretch the robot arm made up of the points $f_{n-1},\ldots,f_{n-(d-1)}$ into a straight line, pointing in the direction of $-e_0$. Recall we write
\begin{eqnarray*}
 f(z_1,\ldots,z_k)&=&(f_1,\ldots,f_{n-2},f_{n-1})\,\,\,\in \,\,\, \mathcal{C}_d(\ell^{n-1}_-)\times D^{d-1}
\end{eqnarray*}
so that each $f_j\in S^{d-1}$ for $j=1,\ldots,n-2$, with $f_{n-1}\in D^{d-1}\subset S^{d-1}$ in a small disc centered at $-e_0$. In fact, $f_1,\ldots, f_{n-d}\in S^0$, and the $0$-th coordinate (the real coordinate in $\R^d=\R\times \C^k$) being negative for all $f_{n-1},\ldots,f_{n-(d-1)}$.

We can ignore the last coordinate and simply think of $(f_1,\ldots,f_{n-2})\in \mathcal{C}_d(\ell^{n-1}_-)$.

Denote by $g$ the composition of $f$ with the projection $p\colon \mathcal{C}_d(\ell)\to (S^{d-1})^{d-2}$ to the coordinates $(f_{n-(d-1)},\ldots,f_{n-2})$. Note that $\ell_{n-(d-1)}f_{n-(d-1)}+\cdots+\ell_{n-2}f_{n-2}=-ce_0$ for some fixed $c>0$. We can think of these $d-2$ coordinates as a robot arm starting at the origin and ending at $-ce_0$. We want to stretch out this robot arm until all coordinates point to $-e_0$. To do this we use the flow $\Phi$ of the standard gradient of the height function on $S^{d-1}$ which has $-e_0$ as its maximum and $e_0$ as its minimum. Consider $G\colon S^{d-2} \to [0,\infty) \to (S^{d-1})^{d-3}$ given by projecting $g$ down to the coordinates $(f_{n-(d-1)},\ldots,f_{n-3})$ and then applying the flow $\Phi$ to each of the $d-3$ coordinates. As we continue to flow, each coordinate approaches $-e_0$. We thus get an induced map $\bar{G}\colon D^{d-1} \to (S^{d-1})^{d-3}$ such that $\bar{G}(0)=(-e_0,\ldots,-e_0)$. Denote by $G_j(z)\in S^{d-1}$ the $j$-th coordinate of $G(z)$ 
for $z\in D^{d-1}$. Then there is a unique element $G_{d-2}(z)\in S^{d-1}$ whose $0$-coordinate is negative, and which ensures that
\begin{eqnarray*}
 \ell_{n-d+1}G_1(z)+\cdots + \ell_{n-3}G_{d-3}(z)+ \ell_{n-2}G_{d-2}(z) &=& -c(|z|)e_0
\end{eqnarray*}
where $c\colon [0,1] \to (0,\infty)$ is a monotonely decreasing function. Note that $G_{d-2}(z)$ exists by elementary geometry, and by the Implicit Function Theorem it depends smoothly on $z$. As the flow can be chosen to be invariant under the $\SO(d-1)$-action on $S^{d-1}$, the function $c$ only depends on $|z|$.

We may think of the points $(G_1(z),\ldots,G_{d-2}(z))\in (S^{d-1})^{d-2}$ as a robot arm depending on $z\in D^{d-1}$ which starts at the origin and has endpoint on the negative real axis in $\R\times \C^k$. Also, on $S^{d-2}$ this agrees with $g$. We need to extend this robot arm to a map $f\colon D^{d-1} \to \mathcal{C}_d(\ell^{n-1}_-)$. To do this, note that the coordinates $(f_1,\ldots,f_{n-d})$ are all in $S^0$. Note that $n-d\geq 3$ and there is at least one element $f_m$ with $f_m=-e_0$ (this is an element of the set $K$). Recall that $K\cup \{n-(d-1),\ldots,n-1\}$ is $\ell$-long, but removing any element of $K$ makes this set $\ell$-short. Now let $\gamma\colon [0,1] \to (S^{d-1})^{n-d}$ with $\gamma(0)=(f_1,\ldots,f_{n-d})$, so that in the $m$-th coordinate the point $-e_0$ is rotated into $e_0$ along $S^1\subset S^{d-1}$, and so that
\begin{eqnarray*}
 \ell_1\gamma_1(t)+\cdots \ell_{n-d}\gamma_{n-d}(t) &=& d(t)e_0
\end{eqnarray*}
with $d\colon [0,1]\to (0,\infty)$ a strictly monotone increasing map. Note that we only have to modify two extra coordinates beside $m$, so this is easily done. We can also do this so that
\[
 p(\gamma_1(t)),\ldots,p(\gamma_{n-d}(t))
\]
have rank at least $3$ for all $t\in (0,1)$ and for any projection $p\colon \R\times \C^k \to \R\times \C$ to one of the $\C$-coordinates of $\C^k$ (and keeping the $\R$-coordinate). For this we should rotate the $m$-th coordinate diagonally through $\C^k$ rather than through $\C\times \{0\}$, and similarly with the other coordinates.

We want to combine $\gamma$ and $G$ to the map $f\colon D^{d-1}\to \mathcal{C}_d(\ell^{n-1}_-)$, using the formula
\begin{eqnarray*}
 f(z)&=& (\gamma(s(|z|)),G_1(z),\ldots,G_{d-2}(z)).
\end{eqnarray*}
For some map $s\colon [0,1] \to [0,1]$ with $s(1)=0$. Note that both $\gamma$ and $G$ both end up on $\R\times \{0\}$ when adding up the coordinates, so we need to choose $s(|z|)$ with
\begin{eqnarray*}
 \ell_n-\ell_{n-1}+c(|z|)+d(s(|z|))&=& 0.
\end{eqnarray*}
Since $d$ is invertible on its image, and $\ell_{n-1}-c(|z|)-\ell_n$ in this image by the choice of the set $K$, we can find this $s$.

Because of the $(S^1)^k$-equivariance of $f$ on $S^{d-2}$ (note that $G$ is still equivariant on the interior of $D^{d-1}$, but $\gamma$ is not), we get the induced map $F\colon \CP^k \to \mathcal{M}_d(\ell)$ after restricting to the $S^1$-diagonal-action. We clearly have $F(\CP^k)\cap \mathcal{M}_d(\ell^{n-1}_-)=\{F(0)\}$, where $0\in \CP^k$ corresponds to $0\in D^{d-1}$. Note however that due to our construction, the rank of $F(0)$ is $3$, so it does not represent a regular point of $\mathcal{M}_d(\ell)$. In order to fix this, let us analyze ranks of images in more detail.

Consider the restriction $f|\colon S^{d-2} \to \mathcal{C}_d(\ell^{n-1}_-)$ which induces a map $F|\colon \CP^{k-1}\to \mathcal{M}_d(\ell^{n-1}_-)$. For any subset $A\subset \{1,\ldots,k\}$ with $A\not=\emptyset$ we have natural subspace $\CP_A\cong \CP^{|A|-1}$ consisting of those elements $[z_1:\cdots:z_k]$ with $z_i=0$ if $i\notin A$. These subspaces form a natural stratification of $\CP^{k-1}$
\[
 \CP^{k-1}_0 \subset \CP^{k-1}_2 \subset \cdots \subset \CP^{k-1}_{2k-4} \subset \CP^{k-1}_{2k-2}=\CP^{k-1}
\]
and by the construction of $f|$ the restriction $F|$ is a stratified map $F|\colon \CP^{k-1}\to \mathcal{M}_d(\ell^{n-1}_-)$ with
\[
 F|(\CP^{k-1}_{2i}-\CP^{k-1}_{2i-2})\,\,\,\subset \,\,\, \mathcal{N}_{2i+3}(\ell^{n-1}_-)-\mathcal{N}_{2i+2}(\ell^{n-1}_-) \mbox{ for }i\leq k-2
\]
and
\[
 F|(\CP^{k-1}-\CP^{k-1}_{2k-4})\,\,\,\subset \,\,\, \mathcal{M}_d(\ell^{n-1}_-)-\mathcal{N}_{d-2}(\ell^{n-1}_-).
\]
Note that $\CP^k$ has a similar stratification as $\CP^{k-1}$ using $A\subset \{1,\ldots,k+1\}$ with $A\not=\emptyset$, and the choice of $\gamma$ ensures that we also have a stratified map $F\colon \CP^k \to \mathcal{M}_d(\ell)$ with
\[
 F(\CP^k_{2i}-\CP^k_{2i-2})\,\,\,\subset \,\,\, \mathcal{N}_{2i+3}(\ell^{n-1}_-)-\mathcal{N}_{2i+2}(\ell^{n-1}_-) \mbox{ for }i\leq k-1
\]
and
\[
 F(\CP^k-\CP^k_{2k-2})\,\,\,\subset \,\,\, \mathcal{M}_d(\ell^{n-1}_-)-\mathcal{N}_{d-2}(\ell^{n-1}_-).
\]
We want to change $F$ by slightly changing $f\colon D^{d-1}\to \mathcal{C}_d(\ell^{n-1}_-)$ and we do this by using a diffeomorphism $\varphi$ of $D^{d-1}$ which sends $0$ to a point near $0$ that is send to a point of rank $d-1$ under $f$, and is the identity outside a small neighborhood of $0\in D^{d-1}$. The induced map $F\colon \CP^k\to \mathcal{M}_d(\ell)$ still has exactly one point in the intersection with $\mathcal{M}_d(\ell^{n-1}_-)$, namely $F(0)$, but this time the rank is $d-1$. We can alter $f\circ \varphi^{-1}$ to make it constant near $0$, thus ensuring that the intersection is also transverse.

By Lemma \ref{lem_strata} we get our element $Y=F_\ast[\CP^k]\in I^{\mathbf{p}_1'}H_{d-1}(\mathcal{M}_d(\ell))$ and it satisfies the conditions required in Lemma \ref{lem_geoconst}.

\end{document}